\documentclass[12pt,reqno]{amsart}
\usepackage{amsmath,amssymb}
\usepackage{amsmath,amstext,amssymb,amsopn,amsthm}
\usepackage{amsmath,amssymb,amsthm, graphicx}
\usepackage[mathscr]{eucal}

\newtheorem{theorem}[equation]{Theorem}
\newtheorem{prop}[equation]{Proposition}
\newtheorem{lemma}[equation]{Lemma}
\newtheorem{cor}[equation]{Corollary}
\newtheorem{corollary}[equation]{Corollary}

\theoremstyle{remark}
\newtheorem{remark}[equation]{Remark}
\theoremstyle{definition}

\numberwithin{equation}{subsection}

\newtheorem{example}[equation]{Example}

\newcommand{\R}{\mathbb{R}}

\newcommand{\di}{\operatorname{div}}
\newcommand{\grad}{\operatorname{grad}}
\newcommand{\HH}{\mathcal{H}}
\newcommand{\Z}{\mathcal{Z}}
\newcommand{\V}{\mathcal V}
\newcommand{\Ph}{\mathcal P}
\newcommand{\LL}{\mathcal L}
\newcommand{\Vol}{\operatorname{Vol}}
\newcommand{\supp}{\operatorname{supp}}
\newcommand{\curl}{\operatorname{curl}}
\title[Inverse electrostatic and elasticity problems]{On a class of inverse electrostatic and elasticity problems}
\author{Andrei Artemev}
\address{Department of Mechanical and Aerospace Engineering, Carleton University, 1125 Colonel By Drive, Ottawa, ON, Canada, K1S 5B6}
\email{aartemev@connect.carleton.ca}
\author{Leonid Parnovski}
\address{Department of Mathematics, University College London, 
Gower Street, London WC1E 6BT, UK}
\email{leonid@math.ucl.ac.uk}
\author{Iosif Polterovich}
\address{D\'epartement de math\'ematiques et de
statistique, Universit\'e de Montr\'eal, C. P. 6128, Succ.
Centre-ville, Montr\'eal, Qu\'ebec,  H3C 3J7,  Canada}
\email{iossif@dms.umontreal.ca} 
\begin{document}
\begin{abstract}
We study the  inverse electrostatic and elasticity problems associated with Poisson and
Navier equations.  The uniqueness of solutions  of these problems is proved for piecewise constant electric charge and internal stress distributions
having a checkered structure: they are constant  on rectangular blocks. Such distributions appear naturally in practical applications.  We also 
discuss computational challenges arising in the numerical implementation of our method. 
\end{abstract}
\maketitle
\section{Introduction and main results}
\subsection{Direct electrostatic and elasticity problems}
\label{1.1}
The {\it Poisson} and {\it Navier} (also known as {\it Lam\'e}\,) partial differential equations of elliptic type are commonly used for direct problems in electrostatics and elasticity theory. Let $\Omega$ be a bounded domain in $\mathbb{R}^n$, $n=2,3,\dots$; the cases $n=2,3$ are the most physically interesting.
It is assumed throughout the paper that   $\Omega$ has a   piecewise smooth boundary $\Gamma=\partial \Omega$.

In the direct formulation of the electrostatic problem, the Poisson equation   for a real valued function $u(x)$,  called
the {\it electric potential distribution},
\begin{equation}
\label{Poisson}
\Delta u = f(x),
\end{equation}
is solved in  the domain $\Omega$
  with a known distribution of the {\it electric charge density}, $-f(x)$, and with definite boundary conditions set on $\Gamma$. The boundary conditions may be formulated either in the form of potential values ({\it Dirichlet conditions})
\begin{equation}
\label{dir:elec}
u|_{\Gamma}=\phi_1
\end{equation}
or in terms of the electric field ({\it Neumann conditions}),
\begin{equation}
\label{neum:elec}
(\nabla u, \nu)|_{\Gamma} =\phi_2.
\end{equation}
Here $\nu=(\nu_1,\dots,\nu_n)$ is the unit outer normal vector to $\Gamma$ and $(\cdot,\cdot)$ is the Euclidean scalar product in ${\mathbb R}^n$.
Different parts of   $\Gamma$ may have different types of boundary conditions, and at any part of $\Gamma$  only one boundary condition may be set (which may be a linear combination of Dirichlet and Neumann conditions), so that the
problem is not overconstrained.

\smallskip

The direct formulation of the elasticity problem is described by the Navier equation
\begin{equation}
\label{Navier}
\Delta U +\alpha\grad \di U =F(x),
\end{equation}
for a vector-valued function $U:\Omega \to \mathbb{R}^n$, called
the {\it displacement field}.
Here $F(x)= -\frac{2(1+\mu)}{E}\mathcal{F}(x)$, where $\mathcal{F}$ is  the distribution of {\it body forces}, $\mu$ is the  {\it Poisson's ratio}, $E$ is the {\it Young's modulus} and parameter $\alpha$ is related to the Poisson's ratio by  formula
$\alpha=\frac{1}{1-2\mu}.$   The body $\Omega$ is assumed to be elastically isotropic.
The equation \eqref{Navier} is solved for the known distribution of body forces in  $\Omega$  and the boundary conditions at  $\Gamma$  defined for displacements  ({\it Dirichlet conditions})
\begin{equation}
\label{dir:elast}
U|_\Gamma=\Phi_1
\end{equation}
or for  traction forces ({\it Neumann conditions}),
\begin{equation}
\label{neum:elast}
(\sigma, \nu)=\Phi_2.
\end{equation}
Here $\sigma$ is a $(0,2)$--tensor (called the {\it stress tensor}), whose components are  related to the components of the displacement gradient through Hooke's law:
$$\sigma_{ij}(U)=(\alpha-1)\, \delta_{ij}\, \di U+
\partial U_i/\partial x_j + \partial U_j/\partial x_i,\,\,\,\, $$
$i,j=1,\dots,n$,  where $\delta_{ij}$ is the Kronecker symbol. Note
that the Hooke's law is given above in dimensionless form
corresponding to the unit value of the shear modulus. The scalar product
$(\sigma, \nu)$ is  a vector in ${\mathbb R}^n$ with the components
$$\sum_{j=1}^n \sigma_{ij}(U)\,\nu_j, \,\,\,i=1,\dots, n.$$ At any
part of $\Gamma$ the boundary condition can be specified for
the displacement, or for the traction force, or for a linear combination
between displacements and traction forces. As in the  direct
electrostatic problem,  only one boundary condition can be
assigned at any part of $\Gamma$. The attempt to define
simultaneously two different types of boundary conditions at the
same part of  $\Gamma$ (i.e. to impose the {\it Cauchy conditions}
\cite[chapter 6]{MF} corresponding to the overconstraining of the
system) may lead to the loss of the solution.

The properties of the direct electrostatic  and elasticity problems have been studied intensively for almost two centuries. It is well-known that  problems \eqref{Poisson} and \eqref{Navier} have unique solutions
under the Dirichlet boundary conditions \eqref{dir:elec} and \eqref{dir:elast}, respectively. For Neumann boundary conditions,  solutions exist under
additional  assumptions $\int_\Omega f dx =\int_{\Gamma} \phi_2 ds$ and $\int_\Omega F dx = \int_{\Gamma} \Phi_2 ds$,
respectively, and are unique up to additive constants  (see \cite{Ja, TG}).

There are various analytical and numerical methods to find solutions of the boundary value problems for Poisson and Navier equations. Most of the numerical methods developed for these problems are based on finite difference approximations  \cite{Hi, MG, Sa, St},  finite element analysis \cite{Ba, Sa, CS}
and Fourier transform  \cite{Du, Kh}.The finite element method has become a dominant approach to solving the elasticity problems,  with the exception of the microelasticity analysis for strain interactions in microstructures, where the Fourier transform is still used intensively. All major numerical techniques are still used for the electrostatic (or magnetostatic) and electromagnetic  problems.


%

\subsection{Inverse problems for Poisson and Navier equations}
\label{sec:unique}
In the present paper we consider the following inverse electrostatic and elasticity problems:  the  charge distributions or  the internal body forces are not known, and the states of the system, i.e., the functions $f$ or  $\mathcal{F}$, should be determined from the boundary conditions \eqref{dir:elec}--\eqref{neum:elec} or, respectively,  \eqref{dir:elast}--\eqref{neum:elast}.   Note that it is assumed that both Dirichlet and Neumann data are obtained from the boundary measurements. Such problems have attracted much interest in the recent years among physicists and engineers (see section \ref{sec:disc} and references therein).

The problems described above can be viewed as examples of inverse problems of potential theory \cite{Isak3}. These problems are  different from the 
Calder\'on's inverse conductivity and elasticity problems, for which the coefficients of the left--hand sides of the equations, rather than the right--hand sides, are unknown and have to be determined from the boundary data (see, for example,  \cite{Ca, Uh1, Uh2,  AMR,  Isak2}).

The inverse electrostatic problem formulated above is  closely related to  the  inverse gravimetry problem\footnote{We thank Leonid Polterovich for bringing this link  to our attention.} that has important applications to geophysics and has been intensively studied for many years (see, for instance,  \cite{Isak1, Isak2,  FM} and references therein). 

In the present subsection we collect some general results on uniqueness of solutions of inverse problems for Poisson and Navier equations.  Essentially, they are well-known (see, for example, \cite{BSB}). We present their proofs  in subsection \ref{sec:harm} for the sake of completeness.



Let, as before, $\Omega \subset \mathbb{R}^n$ be a Euclidean domain with piecewise smooth boundary $\Gamma$.
Consider the following overdetermined boundary value problem for the Poisson equation:
\begin{equation}
\label{problem}
 \Delta u = f, \quad x \in \Omega,
\end{equation}
$$u|_{\Gamma} =0, \quad \left.(\nabla u, \nu)\right|_{\Gamma}=0.
$$
Let $H(\Omega)$ be the space of {\it harmonic functions} on $\Omega$.
Denote by $Z(\Omega)$ its orthogonal complement in $L^2(\Omega)$.
We have the following
\begin{theorem}
\label{harm} A nonzero solution of problem \eqref{problem} exists
if and only if  $f \in Z(\Omega)$.
\end{theorem}

Let $V \subset L^2(\Omega)$ be a linear subspace. We say that the inverse electrostatics problem  possesses a {\it uniqueness property} for charge distributions in $V$ if
for any  two solutions  $u$ and $w$ of the Poisson equations $\Delta u =f$ and $\Delta w= g$ in $\Omega$  with $f, g \in V$, the equalities
$u|_{\Gamma}=w|_{\Gamma}$ and $ (\nabla u,\nu)|_{\Gamma} =(\nabla w, \nu)|_{\Gamma}$ imply $f \equiv g$. Since $V$ is a linear subspace of $L^2(\Omega)$ and the Poisson equation is also linear, this is equivalent to saying that for any nonzero $f \in V$,  problem \eqref{problem} does not have
a solution.  Therefore, Theorem \ref{harm} implies the following
\begin{cor}
\label{cor1}
The inverse electrostatics problem possesses a uniqueness property for charge distributions in a linear subspace $V(\Omega) \subset L^2(\Omega)$ if and only if $V(\Omega) \cap Z(\Omega) =0.$
\end{cor}
\begin{remark} It follows immediately from Corollary \ref{cor1} that the inverse electrostatics problem possesses a uniqueness property if $V(\Omega)
\subset H(\Omega)$. For instance, this is true if $V(\Omega)$ is the space of {\it linear functions} on~$\Omega$.
\end{remark}
Similar results hold for the inverse elasticity problem. Consider an
overdetermined problem for the Navier equation:
\begin{equation}
\label{problem2}
\Delta U + \alpha \, \grad \di U = F, \quad x \in \Omega,
\end{equation}
$$U|_{\Gamma}=0,  \quad  \left. (\sigma, \nu)
\right|_{\Gamma}=0.
$$
Let
$$\mathcal{L}=\Delta  + \alpha \, \grad \di $$
be the Navier operator acting on vector-valued functions $U:\Omega \to
{\mathbb R}^n$.
Denote by $\HH(\Omega)$ the kernel of $\mathcal{L}$ (i.e., the
analogue of harmonic functions for the Navier operator) and by $\Z(\Omega)$ its orthogonal
complement in $L^2(\Omega, {\mathbb R}^n)$.
\begin{theorem}
\label{harm2}  A nonzero solution of  problem
\eqref{problem2} exists if and only if $f \in \Z(\Omega)$.
\end{theorem}
Let $\V(\Omega) \subset L^2(\Omega, \mathbb{R}^n)$ be a linear subspace. We say that  the inverse elasticity problem possesses a uniqueness property for internal stress
distributions in $\V$ if  for any  two solutions  $U$ and $W$ of the Navier equations $\mathcal{L}U=F $ and $\mathcal{L}W=G$ in $\Omega$  with $F, G \in \V$, the equalities
$U|_{\Gamma}=W|_{\Gamma}$ and $ (\sigma_U,\nu)|_{\Gamma} =(\sigma_W, \nu)|_{\Gamma}$ imply $F \equiv G$. Here
 $\sigma_U$ and $\sigma_W$ denote the stress tensors associated with $U$ and $W$, respectively.
Since the Navier equation and the space $\V(\Omega)$ are linear,  Theorem~\ref{harm2} immediately implies
\begin{corollary}
\label{cor2}
The inverse elasticity problem possesses a uniqueness property for internal stress distributions in a linear subspace $\V(\Omega) \subset L^2(\Omega,\mathbb{R}^n)$
if and only if $\V(\Omega) \cap \Z(\Omega)=0$.
\end{corollary}


\subsection{Checkered distributions}
\label{sec:check}
In practical applications one can often assume that the distributions of electric charge as well as of the internal stress  have  a certain structure.
For example, the geometry of the charge density distribution in the electronic component can be dictated by the structure of the component. Such structures often consist of elements with rectangular shape (see, for instance, \cite{LK, XCS}). Let us also note that similar structures appear naturally in  geophysics \cite{Tsch}. This motivates the following definition.

We say that the set $\Pi \subset \mathbb{R}^n$ is a {\it box} if  $\Pi
=[a_1,b_1)\times \dots \times [a_n,b_n)$, $a_i<b_i, i=1,\dots, n$.
Denote by  $V_c(\Pi) \subset L^2(\Pi)$ a linear subspace generated
by the characteristic functions of all boxes contained in $\Pi$. 
Elements of $V_c(\Pi)$ are called {\it checkered functions}.
Equivalently, a function $f\in L^2(\Pi)$ is checkered if $\Pi$ can
be represented as a  finite union of disjoint boxes,
$\Pi=\Pi_1\sqcup\dots \sqcup\Pi_N$, such that $f|_{\Pi_i} \equiv
{\rm const}$, $i=1,\dots N$ (such a
representation is clearly not unique).
Note that the subspace $V_c(\Pi)$ is dense in $L^2(\Pi)$.
\begin{theorem} \label{main}  Let $u$ and $w$ be solutions of the Poisson equations $\Delta u =f$ and $\Delta w=g$ in the interior of the box $\Pi \subset \mathbb{R}^n$  with $f, g \in V_c(\Pi)$.  If $u|_{\partial \Pi}=w|_{\partial \Pi}$ and $ (\nabla u,\nu)|_{\Gamma} =(\nabla w, \nu)|_{\partial \Pi}$, then $f \equiv g$.
\end{theorem}
 In other words, the inverse electrostatics problem on $\Pi$ possesses a uniqueness property for electric charge distributions given by checkered functions.

\begin{remark}
In the context of the inverse gravimetry problem,  the right--hand side of equation \eqref{Poisson} should be understood as  the mass density and the function $u$ as the gravitational potential. Therefore, Theorem~\ref{main} can be reformulated as follows: the inverse gravimetry problem possesses
a uniqueness property for mass distributions given by checkered functions.  To our knowledge, distributions of this type have not been previously studied  for the inverse gravimetry problem. Uniqueness results for other types of mass distributions 
could be found in \cite[Corollary 4.2.3]{Isak1} and  \cite[Theorem 2.1]{Isak3}.
\end{remark}

An analogue of Theorem \ref{main}  holds also for the inverse elasticity problem. Denote by
$\V_c(\Pi) \subset L^2(\Pi, {\mathbb R}^n)$ a linear
subspace generated by functions $F=(f_1,f_2,\dots,f_n)$, where $f_i
\in V_c(\Pi)$, $i=1,\dots, n$.
\begin{theorem} \label{main2} Let $U$ and $W$ be solutions of the Navier equations  $\Delta U +\alpha \grad \di U=F$ and
$\Delta W +\alpha \grad \di W=G$   in the interior of the box $\Pi \subset \mathbb{R}^n$  with $F, G \in \V_c(\Pi)$.  Suppose that $U|_{\partial \Pi}=W|_{\partial\Pi}$ and $ (\sigma_U,\nu)|_{\partial \Pi} =(\sigma_W, \nu)|_{\partial \Pi}$, where $\sigma_U$ and $\sigma_W$ are the stress tensors associated with $U$ and $W$, respectively.  Then $F \equiv G$.
\end{theorem}
In other words, the inverse elasticity  problem on $\Pi$ possesses a uniqueness property for  internal stress distributions with components given by
checkered functions. 

Theorems \ref{main} and \ref{main2} are proved using Corollaries \ref{cor1} and \ref{cor2}, see section \ref{sec:checkered}.

\subsection{Discussion}
\label{sec:disc}
The interest in  the  inverse problems considered in the present paper arises from a number of practical applications. For example, in microelectronics, the observation of the internal voltage distribution in a device can be very important for the testing and diagnostic of devices under development  \cite{LK, BS}.

The inverse elasticity problem naturally appears in the analysis of {\it residual stresses} \cite{Wi1, Wi2}.  These  stresses are produced in the materials as a result of non-uniform deformation during forming, heat treatment and welding processes. The effect of a residual stress field is similar to the effect of an internal force distribution, and one can be converted into the other.  Modern experimental methods, such as Scanning Probe Microscopy \cite{KB, GAT}, can be used to obtain data on the electric potential and electric field at the surfaces of the component \cite{Pr}. 
Digital image correlation \cite{CRS} can be applied to study the displacement distribution. These methods allow to obtain the Cauchy boundary conditions for electrostatic or elastic problems corresponding to real objects or components with high accuracy and fine resolution. The important question is to which extent such information can be used to find the charge (and the potential) or the internal force distributions inside the body,  and whether the corresponding inverse problems  have unique solutions. For simple distributions of internal charges or residual stresses the inverse problems can be solved easily (for example,  for a 2-D distribution of charges in a thin layer or 1-D distribution of residual stresses with a single significant stress component).
However, for general charge and internal stress distributions the issue becomes quite difficult.  While Theorems \ref{main} and \ref{main2} give a complete mathematical solution of the inverse electrostatics and elasticity problems for checkered distributions, from the viewpoint of practical applications these results are far from satisfactory, see section \ref{sec:chal}.
%


\subsection{Non-uniqueness of solutions: an example}
\label{sec:example}
One may ask whether
the analogues of Theorems \ref{main} and \ref{main2} hold for other, non-checkered,
electric charge and internal stress distributions.  Below we provide an example of a natural class
of distributions for which the solutions of the inverse problems are not unique.  Similar examples are well-known for the
inverse gravimetry problem (see \cite{Isak3}).

Let $S=S(r_1,r_2)$ be a {\it spherical layer} centered at the
origin, that is $S(r_1,r_2)=\{r_1\le |x|<r_2\}$ for some $r_2>r_1 \ge 0$.
Denote by $V_\sigma(S) \subset L^2(S)$ the linear subspace
generated by characteristic functions of  spherical
layers centered at the origin. In other words, $f\in V_\sigma(S)$ if and only if there exists a decomposition of
$S$ into a disjoint union of spherical layers $S=S_1\sqcup\dots \sqcup S_N$, such that $ f|_{S_i} \equiv {\rm const}$, $i=1,\dots N$.
We also denote by  $\V_\sigma(S) \subset
L^2(S,{\mathbb R}^n)$ the linear subspace of vector
functions whose components belong to $V_\sigma(S)$.
\begin{theorem}
\label{spherical} Let $S \subset \mathbb{R}^n$ be a spherical layer. Then

(i) $V_\sigma (S) \cap Z(S) \neq ~\{0\}$ and
(ii) $\V_\sigma(S) \cap \Z(S) \neq \{0\}.$
\end{theorem}
Theorem \ref{spherical} is proved in subsection \ref{sec:sph}. Together with Corollaries \ref{cor1} and \ref{cor2},  it immediately implies
\begin{corollary} The solutions of the inverse electrostatics and elasticity problems are not unique in $V_\sigma(S)$ and $\V_\sigma(S)$, respectively.
\end{corollary}

\subsection{Plan of the paper}

Section \ref{sec:checkered} is devoted to the proof of Theorems~\ref{main} and \ref{main2}. In subsection \ref{discr} an auxiliary discretization of the checkered functions is constructed.  In subsection \ref{sec:exp} we introduce a family of harmonic functions given by complex exponentials,
that are used to show that there are no nonzero checkered functions orthogonal to the space of harmonic functions.
Theorem \ref{main} then follows from Theorem \ref{harm}. In subsection
\ref{sec:elast} the above arguments are modified in order to prove Theorem \ref{main2}. Theorems \ref{harm}  and \ref{harm2} as well as Theorem \ref{spherical} are proved in section \ref{sec:other}.

\section{Inverse  problems for checkered distributions}
\label{sec:checkered}
The goal of this section is to prove  Theorems \ref{main} and \ref{main2}. We present the proofs in
three dimensions, which is the most interesting case for
applications. A similar argument works in any dimension $n\ge~2$.

\subsection{Discretization of checkered functions}
\label{discr}
Let $\Omega$ be a box as defined in section \ref{sec:check}.
For any $f \in V_c(\Omega)$, let us construct a function $\tilde f$
supported on a finite number of points. Consider an arbitrary box
$\Pi=[\beta_1^-,\beta_1^+) \times [\beta_2^-,\beta_2^+) \times
[\beta_3^-,\beta_3^+) \subset \mathbb{R}^3$. Set
\begin{equation}
\label{rule} \Upsilon(\chi_\Pi)=\sum_{\sigma_1,\sigma_2,\sigma_3 \in
\pm} \sigma_1\sigma_2\sigma_3 \mathbf{1}_{(\beta_1^{\sigma_1},
\beta_2^{\sigma_2}, \beta_3^{\sigma_3})}
\end{equation}
 Here
$\mathbf{1}_{(x,y,z)}$ is a function that takes value $1$ at the
point $(x,y,z)$ and vanishes elsewhere. The function
$\Upsilon(\chi_{\Pi})$ is supported on the vertices of $\Pi$ and
takes values $\pm 1$ at each vertex. The map $\Upsilon$ can be then
extended by linearity to the whole space $V_c(\Omega)$.

Given a function $f\in V_c(\Omega)$, set $\tilde f = \Upsilon(f)$.
Denote by $\tilde V_c(\Omega)$ the space of functions supported on
finite subsets of $\Omega$.

\begin{prop}
\label{ups} The map $\Upsilon:V_c(\Omega) \to \tilde V_c(\Omega)$ is
injective. Moreover,  there exists a constructive procedure to
recover $f\in V_c(\Omega)$ from the function $\Upsilon(f)=\tilde f$.
\end{prop}

To prove Proposition \ref{ups} we need an  auxiliary lemma below.

 Let $\{(x_l,y_l,z_l)\}_{l=1}^N$ be the
collection of vertices of all the boxes appearing in some
representation of $f$ as a linear combination of characteristic
functions of boxes. We say that a point $(a,b,c) \in \Omega$ is a
{\it node} of the function $f$ if $a=x_i, b=y_j, c=z_k$ for some $1
\le i,j,k \le N$.
A node $v$ is {\it interesting} if
$\tilde f(v) \neq 0$.   We also call a node $v$ {\it artificial} if  there exists a neighborhood of $v$ in which $f$ does not change its value across  a plane passing through $v$ and parallel to one of the coordinate planes. 
It is easy to check that all artificial nodes are not interesting (and, therefore, artificial nodes can not be determined from $\tilde f$), but the
converse is not necessarily true.

\begin{example} Let $\Omega$ be a cube with side $2$ centered at the origin $v=(0,0,0)$.
Let $f$ be a restriction to $\Omega$ of a function which is identically equal to $1$ in 
the positive and the negative octants, and vanishes elsewhere. Then $v$ is not an artificial node, but at the same time  $\tilde f(v)=0$ by \eqref{rule}
and, hence,  $v$ is  not interesting.
\end{example}

\begin{remark} One could view the difference between artificial and non-artificial nodes as follows. Let us  colour $\Omega$ in such a way that points $x,y \in \Omega$ have the same colour if and only
if $f(x)=f(y)$.  Then $\Omega$ can be represented as a disjoint union of sets $\Omega=\sqcup_{j=1}^J \Omega_j$, such that all points in $\Omega_j$,$ j=1,\dots J$,  have the same colour, 
and the points in $\Omega_i$ and $\Omega_k$, $i\neq k$ have different colours. 
Each $\Omega_j$ is a not necessarily connected union of boxes.  A node is not artificial if it is a vertex of one of the sets  $\Omega_j$, and artificial otherwise.
\end{remark}

Let $\supp \tilde f=\{(p_l,q_l,s_l)\}_{l=1}^M$ be the set of
interesting nodes. We say that a point $(a,b,c) \in \Omega$ is a
 {\it marked node} if $a=p_i, b=q_j, c=s_k$ for some $1 \le i,j,k \le
 M$.
 
 Note that the properties of being
a marked node or an interesting node do not depend on the choice of
the representation of $f$.
\begin{lemma}
\label{marked} The set of all marked nodes contains the set of
all non-artificial nodes.
\end{lemma}
\begin{proof} Without loss of generality, suppose that the node
$(0,0,0)$  is not marked. This means that among interesting nodes
there are either no points with $x=0$, or with $y=0$, or with $z=0$.
In each case, the corresponding plane (say, $x=0$) does not contain
interesting nodes. Let us show that the function $f$ does not change
its value across this plane. This would mean that
all nodes contained in this plane are  artificial, including
$(0,0,0)$.

Consider a  decomposition of $\Omega$ into boxes,  such that the set of all their vertices coincides with the set of all nodes of $f$ (this could be achieved by
constructing planes through each node parallel to the coordinate planes).  It follows from the definition of a node that $f$ is constant on each of these boxes.  Take one of the corner nodes belonging to the plane  $x=0$ (i.e. a node lying on one of the edges of $\Omega$). At each such node
at most two boxes meet. Therefore, if this  node is not
interesting, the values of $f$ at the boxes adjacent to it are equal and
hence the node is artificial.   Note that by  formula \eqref{rule}, the  total contribution of these two boxes to the value of $\tilde f$ at any other node lying on the plane $x=0$  is zero.
Let us throw away these two boxes and pick another node where at most two of the remaining boxes meet.  Again, the value of $\tilde f$ at this node is zero and hence
the values of $f$ at the boxes adjacent to it  coincide. Therefore, this node is also artificial. We repeat the procedure until all boxes adjacent to the plane $x=0$ are
thrown away.  At each step we get artificial nodes only. This
completes the proof of the lemma.
\end{proof}
Let us now prove Proposition \ref{ups}. The proof is based on a similar inductive argument as above.  We start at a corner box,
on which  formula \eqref{rule} allows us to reconstruct  in an unambiguous way the value of $f$   from the value of $\tilde f$ on the corresponding corner vertex. We remove that box, move to an adjacent one
and repeat the procedure.  A similar  approach will be used again in the proof of Proposition \ref{prop3}. 
\begin{proof} As follows from Lemma \ref{marked}, knowing $\supp \tilde f$
allows us to construct a decomposition of $\Omega$ into boxes, whose
vertices include all non-artificial nodes. We know the values of
$\tilde f$ at each vertex of these boxes. Let us now reconstruct the
value of $f$ at each of the boxes using the following inductive
procedure. Start with a vertex that is also a vertex of $\Omega$,
and take the box that contains it (there is a unique box with this
property). Since there are no other boxes containing this vertex, by
\eqref{rule}, the value of $\tilde f$ at this vertex determines the
value of $f$ at the box. We subtract the contribution of this box to
$\tilde f$, throw away this box and take one of the new corner
vertices, at which at most two of the remaining boxes  meet.  At each step of this procedure we determine the value of
$f$ on the corner box, and reduce the number of boxes by one. Since
the number of boxes is finite, eventually we will determine the
value of $f$ on each box. 
\end{proof}

\subsection{Exponential functions}
\label{sec:exp}
 Let
$$e=e(x)=e(\alpha,\Theta,\Psi;x)=e^{\alpha(\Theta,x)+i\alpha(\Psi,x)}$$
be a function of the variable $x\in \mathbb{R}^3$, depending on the
parameters $0\neq \alpha \in \mathbb{R}, \Theta\in \mathbb{R}^3$,
$\Psi \in \mathbb{R}^3$, such that $(\Theta, \Psi)=0$,
$|\Theta|=|\Psi|=1$. It is easy to check that $e(x) \in
H(\mathbb{R}^3)$.

We say that a pair of vectors $(\Theta,\Psi)$ is {\it admissible} if
the plane it generates is not  orthogonal to any of the coordinate
axes. Set
\begin{equation}
\label{int:vol}
 P_f(\alpha, \Theta, \Psi)=(f,
e(\alpha,\Theta,\Psi;x))=\int_\Omega f(x)
e^{\alpha(\Theta,x)+i\alpha(\Psi,x)}\,dx.
\end{equation}
\begin{lemma}
\label{lemma:prod} Let $\{v_j\}$ be the set of interesting nodes of
$f\in V_c(\Omega)$. Then, for any $\alpha \neq 0$ and any admissible
pair $(\Theta, \Psi)$ we have:
\begin{equation}
\label{prod} P_f(\alpha,\Theta,\Psi)=C \sum_{v_j} \tilde
f(v_j)\,e(\alpha,\Theta, \Psi;v_j),
\end{equation}
where
\begin{equation}
\label{const} C=C(\alpha,\Theta, \Psi)=\frac{1}{\alpha^3}
\prod_{l=1}^3 \frac{1}{\Theta_l+i\Psi_l}.
\end{equation}
\end{lemma}
Note that the constant $C$ is well-defined for any admissible pair
$(\Theta, \Psi)$.
\begin{proof} The result follows from \eqref{rule} by a direct computation of the triple
integral \eqref{int:vol}. 
\end{proof}
Since any function $e(x)$ is harmonic, the right-hand side in
\eqref{prod} can be computed using the boundary data $\phi_1,\phi_2$
of problem \eqref{problem} by Green's formula:
$$P_f(\alpha, \Theta, \Psi)=\int_{\Gamma} \left(e(x)\phi_2
-\frac{\partial e(x)}{\partial n}\,\phi_1 \right)\, ds
$$
\begin{prop}
\label{prop3} Knowing the value of $P_f(\alpha, \Theta, \Psi)$ for
any $\alpha \neq 0$ and any admissible pair $(\Theta, \Psi)$, one
can reconstruct the function $\tilde f$.
\end{prop}
\begin{proof}
Let $K$ be the convex hull of $\supp \tilde f$. It is easy to see
that $K$ is a convex polyhedron; let $w_j$ be its vertices. Then, for
any $\Theta \in \mathbb{R}^3$ and any $\Psi$ chosen in such a way
that the pair $(\Theta, \Psi)$ is admissible,  we have:
\begin{equation}
\label{limsup} \limsup_{\alpha \to \infty}
\frac{\log|P_f(\alpha,\Theta,\Psi)|}{\alpha}=\max_j(w_j,\Theta)=\max_{y\in K}
(y,\Theta),
\end{equation}
where the first equality follows from \eqref{prod} and the second
one from a well-known fact that the maximum of a linear functional
on a convex polyhedron is attained at one of the  vertices. At the
same time, a convex set $K$ can be represented as the intersection
of its supporting half-spaces:
\begin{equation}
\label{326}
K=\cap_{\Theta \in \mathbb{R}^3} \{ x \in \mathbb{R}^3\,|\,
(x,\Theta) \le \max_{y\in K} (y, \Theta)\}.
\end{equation}
Therefore, by \eqref{limsup}, we can recover $K$ and, in particular,
all its vertices $w_j$, $j=1,\dots, N$.

In order to recover the values $\tilde f(w_j)$ we use the following
procedure. Let $\Theta(j)$ be an external unit normal vector to a plane passing through $w_j$ and not intersecting the convex set $K$ (external means here that $\Theta(j)$ points to the half-space not containing $K$). One can easily check that in this case
\begin{equation}
\label{thetacond} (\Theta(j),w_j)>(\Theta(j),w_k)
\end{equation}
 for all $k \neq j$, $k=1,\dots,N$.
 
Choose $\Psi(j)$ in such a
way that the pair $(\Theta(j), \Psi(j))$ is admissible. We have:
\begin{equation}
\label{327}
\tilde f(w_j)=\lim_{\alpha \to
\infty}\frac{P_f(\alpha,\Theta(j),\Psi(j))}{C(\alpha,\Theta(j),\Psi(j))\,
e(\alpha_,\Theta(j),\Psi(j),w_j)}.
\end{equation}
This allows us to  determine the values of $\tilde f$ at all vertices
$w_j$, $j=1,\dots,N$. Using Lemma \ref{lemma:prod} we can subtract
the contributions of these nodes from $P_f(\alpha,\Theta,\Psi)$ and
repeat the procedure. Since at each step the number of nodes
decreases, the number of steps will be finite and at the end we will
recover all elements of $\supp \tilde f$ and the values of $\tilde
f$ at each of these points.
\end{proof}
Combining Propositions \ref{prop3} and \ref{ups} we obtain the proof
of  Theorem \ref{main}.
\begin{remark} The results of this section generalize in a straightforward way to any dimension $n\ge 2$.
Note that in dimension $n=2$ the admissibility assumption can be
omitted, because for any orthogonal nonzero vectors $\Theta, \Psi
\in {\mathbb R}^2$, the denominators in \eqref{const} are
automatically nonzero.
\end{remark}

\subsection{Proof of Theorem \ref{main2}}
\label{sec:elast}
Let us indicate how the proof of Theorem \ref{main}  can be modified in order to
prove Theorem \ref{main2}. Let $F=(f_1,f_2,f_3) \in \V_c(\Omega)$
and let $\widetilde F = (\tilde f_1, \tilde f_2, \tilde f_3)$ be its
discretization in the sense of section \ref{discr}. We say that $v_j
\in \Omega$ is a node of $F$ if it is a node of one of the
functions $f_i$, $i=1,2,3$. As before, consider a harmonic function
$$e:=e(\alpha,\Theta,\Psi;x)=e^{\alpha(\Theta,x)+i\alpha(\Psi,x)}.$$ It
is easy to check that
$$\curl(e(\alpha,\theta,\Psi;x),0,0)=(0,\alpha(\Theta_3+i \Psi_3)e,
-\alpha (\Theta_2+i\Psi_2)e) \in \HH(\Omega).$$ This follows from
the fact that $e$ is harmonic and that $\di \curl =0$. Similarly,
$\curl(0,e(\alpha,\theta,\Psi;x),0) \in \HH(\Omega)$.

Set
$$
\Ph_F^1(\alpha, \Theta, \Psi)=(F,
\curl(e(\alpha,\theta,\Psi;x),0,0)),$$ $$ \Ph_F^2(\alpha, \Theta,
\Psi)=(F, \curl(0,e(\alpha,\theta,\Psi;x),0))$$
(now $(\cdot,\cdot)$ means the natural inner product in $L_2(\Omega,\R^3)$).
Theorem \ref{main2} can be now deduced from Proposition \ref{ups}
and the following analogue of Proposition \ref{prop3}:
\begin{prop}
\label{prop4} Knowing the value of $\Ph^l(\alpha, \Theta, \Psi)$,
$l=1,2$,  for any $\alpha \neq 0$ and any admissible (in the sense
of section \ref{sec:exp}) pair $(\Theta, \Psi)$ one can reconstruct
the function $\widetilde F$.
\end{prop}
\begin{proof}
Similarly to  Lemma \ref{lemma:prod}, we have:
\begin{multline}
\label{prod2}\Ph_F^1(\alpha,\Theta,\Psi)=
\frac{1}{\alpha^2}\prod_{l=1}^3 \frac{1}{\Theta_l+i\Psi_l}
\sum_{v_j} e(v_j)(\tilde f_2(v_j)(\Theta_3+i\Psi_3) - \\ \tilde
f_3(v_j)(\Theta_2+i\Psi_2)).
\end{multline}
When choosing the unit vector $\Psi$, we will make sure
that, apart from the admissibility condition,
the following condition is satisfied:
$$\Theta_3 \Psi_2 -
\Theta_2 \Psi_3 \neq 0.$$ This condition guarantees that if $$\tilde
f_2(v_j)(\Theta_3+i\Psi_3) - \\ \tilde
f_3(v_j)(\Theta_2+i\Psi_2)=0,$$ this automatically implies $\tilde
f_2(v_j)=\tilde f_3(v_j)=0$, and so no term in the sum \eqref{prod2}
may ``accidentally'' vanish.  Therefore, the contribution of each
interesting node will be taken into account. Arguing in the same way
as in the proof of Proposition \ref{prop3} we can recover the convex
hull of $\supp \tilde f_2\cup\supp \tilde f_3$.

Taking $\Ph_F^2(\alpha,\Theta,\Psi)$  instead of
$\Ph_F^1(\alpha,\Theta,\Psi)$ in the argument above we recover the
convex hull of $\supp \tilde f_1\cup\supp \tilde f_3$. Taking a
union of these two sets, we recover the convex hull of $\supp
\widetilde F$.

Let $w_j$ be a vertex of the convex hull of $\supp \widetilde F$.
Choose a  unit vector $\Theta(j)$ satisfying \eqref{thetacond} as in
the proof of Proposition \ref{prop3}. Consider two admissible pairs
$(\Theta(j), \Psi^1(j))$ and $(\Theta_j, \Psi^2(j))$. Using
\eqref{prod2} we can calculate
$$\tilde f_2(w_j)(\Theta_3(j)+i\Psi_3^k(j)) - \tilde
f_3(w_j)(\Theta_2(j)+i\Psi_2^k(j)), \quad k=1,2.$$  We obtain a system
of two linear equations on $\tilde f_3(w_j)$ and $\tilde f_2(w_j)$.
Clearly, we can choose the admissible pairs $(\Theta(j), \Psi^1(j))$ and
$(\Theta_j, \Psi^2(j))$ in such a way that the determinant of this
system is nonzero. Thus, we can compute $\tilde f_3(w_j)$ and
$\tilde f_2(w_j)$.

Applying the same argument to $\Ph_F^2(\alpha,\Theta,\Psi)$, we
compute $\tilde f_1(w_j)$. Therefore, we have computed $\widetilde
F(w_j)$, and this can be done for any vertex of the convex hull of
$\supp \widetilde F$. As in the proof of Proposition \ref{prop3}, we
subtract the contributions of these nodes from
$\Ph_F^1(\alpha,\Theta,\Psi)$ and $\Ph_F^2(\alpha,\Theta,\Psi)$, and
repeat the argument. The process will stop after a finite number of
steps because the number of nodes of $\tilde F$ is finite, and it
decreases at each step. This completes the proof of Proposition
\ref{prop4} and of Theorem \ref{main2}.
\end{proof}
\begin{remark} Instead of using the curl in the proof of Proposition
\ref{prop4}, we could  take $\grad e(\alpha,\theta,\Psi;x)$.
Clearly, $$\grad e(\alpha,\theta,\Psi;x)~\in~\HH(\Omega).$$ In this
case, for each $\Theta(j)$ we need to consider three admissible
pairs $(\Theta(j), \Psi(j;k))$, $k=1,2,3$, in order to get a system
of three linear equations on $\tilde f_k(w_j)$, $k=1,2,3$. The rest
of the proof goes along the same lines as above. The advantage of
this approach is that it works in any dimension, while the curl is
defined only in dimension three.
\end{remark}
\subsection{Computational challenges}
\label{sec:chal} It is convenient to prove Theorems \ref{main} and
\ref{main2} using the exponential functions introduced in subsection
\ref{sec:exp}. In principle, our proof could be presented as an
algorithm that allows to reconstruct in a unique way the solutions
of the Poisson and Navier equations from the corresponding boundary
values.  However, numerical implementation of our approach faces
serious computational difficulties that we describe below. For
simplicity, a $2$--dimensional example is presented.

We have tested  the developed algorithm for the solution of the inverse electrostatic
problem 
on a rectangle $\Omega \subset \mathbb{R}^2$ with boundary $\Gamma$,
containing two rectangular charged areas (Fig. 1). The boundary
conditions corresponding to this problem were obtained using the Green
function method implemented numerically. In other words, we computed
a solution $u$ of the equation $\Delta u =f$ on the whole plane
using Green's function , and calculated numerically its values
as well as the values of its normal derivative on $\Gamma$. Here $f$
is the characteristic function of the total charged area.
\begin{figure}
\begin{center}
\includegraphics[width=6cm]{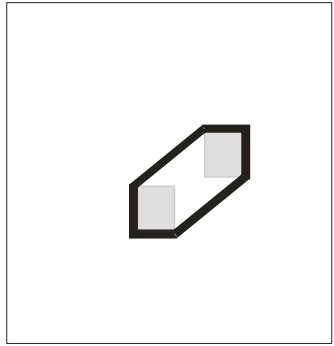}
\caption{Geometry of the test problem. Gray areas represent charge
areas. Thick black line represents the convex envelope of charge
areas. Thin gray line represents the boundary $\Gamma$ of the domain
$\Omega$.}
 \end{center}
\end{figure}

It was found that if $P_f(\alpha,\Theta,\Psi)$  is calculated directly
using the integration over $\Omega$,  then the procedure based on
\eqref{limsup} and \eqref{326} produces the convex hull of the
charged areas with high accuracy. However, when
$P_f(\alpha,\Theta,\Psi)$  is calculated using the integration over
the boundary,  the algorithm based on \eqref{limsup} produces the
convex hull occupying the whole $\Omega$. Such a drastic difference
is the result of small numerical errors in the boundary conditions
determined by the numerical solution of the direct problem,  and also
in the numerical integration over the boundary.

Strong sensitivity to numerical errors
arises from a specific nature of exponential functions. When $\alpha$ is large, these functions
are rapidly increasing in one direction and rapidly oscillating in the orthogonal direction.
As a result, small errors in boundary conditions are multiplied by large factors and the convex
hull estimate based on \eqref{326} and \eqref{327} becomes distorted. For small  values of $\alpha$, the
value of $\log |P_f(\alpha,\Theta,\Psi)|/\alpha$  obtained using the integration over $\Omega$ is close to the one
obtained using the integration over $\Gamma$. However, as $\alpha$ increases these two values diverge.

In order to verify the validity of the computational model,  calculations were performed for the constant charge density: $f \equiv 1$ on $\Omega$.
In this case, the precise value of  $P_f(\alpha, \Theta,\Psi)$ can be computed analytically. For all  values of $\alpha$, the analytical equations produced the
same  $\log |P_f(\alpha,\Theta,\Psi)|/\alpha$  values,  no matter if  the integration was performed  over $\Omega$ or its boundary $\Gamma$.
However, in the  numerical analysis of this problem,  the area and
boundary integration were producing close values for small values  of $\alpha$ and were diverging
for  large $\alpha$.

In practical problems based on experimental data,  boundary conditions are always  obtained with some
errors. Therefore, as the above analysis show, in order to produce an algorithm for the solution of the inverse electrostatics problem
that is numerically implementable, one needs to modify our approach. One possibility would be to find a set of harmonic functions
exhibiting good  behavior from the numerical viewpoint, which could  replace the exponentials in the proof of Theorem \ref{main}.

\section{Necessary and sufficient conditions for uniqueness of solutions}
\label{sec:other}
\subsection{Proofs of uniqueness criteria}
\label{sec:harm}
The goal of this subsection is to prove Theorems \ref{harm} and \ref{harm2}.
Let us start with Theorem \ref{harm}.
To prove necessity, suppose that  $u$ is a solution of problem
\eqref{problem} and let $h \in H(\Omega)$. Then
$$\int_\Omega f \cdot h = \int_\Omega \Delta u \cdot h = \int_\Omega
u \cdot \Delta h =0.$$ Note that the boundary terms in the
integration by parts disappear since $\phi_1=\phi_2=0$.

To prove sufficiency, denote by $G_1(x,y)$ the Green's function of
the Dirichlet boundary value problem in $\Omega$:
$$\Delta_x
G_1(x,y)=\delta(x-y), \,\,\, G_1(x,y)|_{x \in \partial \Omega}=0,$$
and by $G_2(x,y)$ the Green's function for the corresponding Neumann
boundary value problem:
$$\Delta_x
G_2(x,y)=\delta(x-y)-\frac{1}{|\Omega|},\,\,\, (\grad_x G_2(x,y),
\nu)|_{x \in
\partial \Omega}=0.$$ 
Note that the integral over $\Omega$ of the right--hand 
side of the equation above is zero, which is necessary for the existence of a solution of the Neumann problem with zero boundary conditions.

It follows from the definitions of $G_1$ and $G_2$ that
$$(G_1-G_2)(x,y)+ \frac{x_1^2}{2|\Omega|}$$ is a harmonic function of $x$.
Therefore, the assumption $f \in Z$ implies
\begin{equation*}
u(y):=\int_\Omega f(x)G_1(x,y)\,dx = \int_\Omega f(x) G_2(x,y)\, dx -\int_\Omega f(x)\frac{x_1^2}{2|\Omega|}\, dx 
\end{equation*}
for all $y \in \Omega$. Note that the term $$\int_\Omega
f(x)\frac{x_1^2}{2|\Omega|}\,dx$$ is constant and hence of no
importance for the Neumann boundary value problem.  
Let us also remark that  $\int_\Omega f(x) dx =0$ since $f \in Z(\Omega)$.
It is easy to check that, by the properties of $G_1$ and $G_2$, the function $u$
constructed above is a solution of problem \eqref{problem}.  
This completes the proof of Proposition \ref{harm}. \qed

The proof of Theorem \ref{harm2}  is completely analogous. We note
that the existence of Green's functions for the Navier equation with
either Dirichlet or Neumann boundary conditions is well-known~---
see, for instance \cite[section 7.12]{S}.
\subsection{Proof of Theorem \ref{spherical}}
\label{sec:sph} \noindent {\it Proof of (i).} 
Let $A,B,S$ be three
spherical layers centered at the origin such that $A\cup B=S$ and
$A~\cap~B=~\emptyset$. Consider the following linear combination of
characteristic functions of the sets $A$ and $B$:
$$u=\Vol(B)\, \chi_A-\Vol(A)\,\chi_B.$$
By the mean value theorem for harmonic functions we immediately have
$u\in Z(S)$, and this completes the proof of part (i) of the
proposition.

\bigskip

\noindent {\it Proof of (ii)} In order to prove the second part of
the proposition we note that a function $U$ lying in the kernel of
the Navier operator $\LL$ is {\it biharmonic}. Indeed, let $\LL(U) =
\Delta U + \alpha \grad\, \di U =0$. Taking the divergence on both
sides we get $\di (\Delta U)=0$. Here we took into account that $\di
\grad = \Delta$ and that the Laplacian commutes with the divergence.
At the same time, applying the Laplacian to $\LL(U)$ we get
$$\Delta^2\,U + \alpha \Delta\,\grad\,\di U=0.$$ But since $\di
(\Delta U)=0$, the second term vanishes, because $$\Delta \,
\grad\,\di U = \grad \, \di \, \Delta U.$$ Hence, $\Delta^2 U = 0$
and $U$ is biharmonic.

It is well-known that a real valued biharmonic function $f(x)$
satisfies the following mean-value property (see, for example,
\cite{K}):
\begin{equation}
\label{meanvalue} \int_{B(x,r)} f(y) dy = \omega_n r^n f(x) +
\frac{\omega_n \, r^{n+2}}{2(n+2)}\Delta f(x),
\end{equation}
where $B(x,r)$ is a ball of radius $r$ centered at $x \in {\mathbb
R}^n$ and $\omega_n$ is the volume of the unit ball in $\R^n$.

Consider now a  ball $B=B(r_3)$ centered at the origin. Let us
represent it as a union of three sets $B(r_1)\cup S(r_1,r_2) \cup
S(r_2,r_3)$, $r_3 > r_2 > r_1 >0$. Let $F_{a,b}$ be a piecewise
constant vector-valued function taking the values $a$ and $b$ on
$S(r_2,r_3)$ and $S(r_1,r_2$), respectively, and the value one on
$B(r_1)$. Clearly, $F_{a,b} \in V_\sigma$ for all $a,b \in {\mathbb
R}$. Let us show that for any triple $0<r_1 < r_2 < r_3$ there
exists a choice of parameters $a, b$ such that $F_{a,b} \in \Z(B)$.
It can be deduced from the mean-value formula \eqref{meanvalue},
that the inclusion $F_{a,b} \in \Z(B)$ holds if the following system
of equations is satisfied:
$$a(r_3^n-r_2^n) + b (r_2^n-r_1^n) +r_1^n=0,\,\,a(r_3^{n+2}-r_2^{n+2}) + b (r_2^{n+2}-r_1^{n+2})
+r_1^{n+2}=0.
$$
One may check that the determinant of this system does not vanish
for $r_3 > r_2 > r_1 >0$. Indeed, if $r_3 \neq r_2$, the only
positive roots of the determinant considered as a polynomial in
$r_1$ are $r_1=r_2$ and $r_1=r_3$ (there are no other positive roots
because, as can be easily verified, the derivative with respect to
$r_1$ has only one positive root). Hence, there always exists a
unique solution $a,b$ of the system above, and the corresponding
function $F_{a,b} \in \Z(B)$.

This completes the proof of Theorem \ref{spherical}.
\begin{remark}
The proof of part (i) of Theorem \ref{spherical} shows that the
intersection $V_\sigma(S) \cap \Z(S)$  is in fact quite large.
Indeed, it is easy to show that for any partition of $S$ into $k$
concentric spherical layers,  there exists a $(k-1)$--dimensional
linear subspace of functions in $V_\sigma \cap \Z(S)$.

The proof of part (ii) goes through without changes if instead of a
ball $B(r_3)$ one takes a spherical layer
$S=S(r_3,r_0)=S(r_0,r_1)\cup S(r_1,r_2) \cup S(r_2,r_3)$.  It follows
from the proof that for any partition of $S$ into $k$ concentric
spherical layers, there exists a $(k-2)$--dimensional linear space
of functions in $\V_\sigma \cap \Z(S)$.
\end{remark}
\subsection*{Acknowledgments} The authors are grateful to Victor Isakov and Leonid Polterovich for  useful discussions. 
Research of Leonid Parnovski is supported by EPSRC  grant EP/F029721/1.
Research of Iosif Poltero\-vich  is supported by NSERC, FQRNT and Canada Research Chairs program. 

\end{document}